\documentclass[11pt,a4paper]{article}
\usepackage{epsf, graphicx}
\usepackage{latexsym,amsfonts,amsbsy,amssymb}
\usepackage{amsmath,amsthm}
\usepackage{cite}
\usepackage{subfig}
\usepackage{dsfont}
\usepackage{amssymb}
\usepackage[noblocks]{authblk}

\makeatletter

\@addtoreset{equation}{section} \makeatother
\newtheorem{theorem}{Theorem}[section]

\newtheorem{corollary}{Corollary}[section]
\newtheorem{remark}{Remark}[section]

\makeatletter \@addtoreset{equation}{section} \makeatother
\begin{document}
\title{A conforming DG method for linear nonlocal models with integrable kernels}
\date{}
 \author{Qiang Du\footnote{Department of Applied Physics and Applied Mathematics, Columbia University, NY 10027, USA (email: qd2125@columbia.edu)} \quad  and\quad  Xiaobo
Yin\footnote{Corresponding author, School of Mathematics and Statistics \& Hubei Key Laboratory of Mathematical Sciences, Central China Normal University, Wuhan 430079, China (email: yinxb@mail.ccnu.edu.cn)}}
\maketitle

 \begin{quote}
\begin{small}
{\bf Abstract.}\,\, Numerical solution of nonlocal constrained value problems with integrable kernels are considered. These nonlocal problems arise in nonlocal mechanics and nonlocal diffusion. The structure of  the true solution to the problem is analyzed first.  The analysis leads naturally to a new kind of discontinuous Galerkin method that efficiently solve the problem numerically. This method is shown to be asymptotically compatible. Moreover, it  has optimal convergence rate for one dimensional case under very weak assumptions, and almost optimal convergence rate for two dimensional case under mild assumptions.


\end{small}
\end{quote}

\section{Introduction}
Nonlocal models have generated much interests in recent years \cite{du2018icm}.
For example, the peridynamic (PD) model proposed by Silling \cite{silling2000reformulation} is an integral-type nonlocal theory of continuum mechanics which provides an alternative setup to classical continuum mechanics based on PDEs. Since PD avoids the explicit use of spatial derivatives, it is especially effective for modeling problems involving discontinuities or other singularities in the deformation \cite{askari2008peridynamics,du18je,kilic2010coupling,oterkus2012peridynamic,silling2010crack,silling2010peridynamic}.
The nonlocal diffusion (ND) model, described in \cite{du2012analysis} is  another example of integro-differential equations. More recently, mathematical analysis of nonlocal models is also paid more attention, which could be found in \cite{aksoylu2010results,aksoylu2011variational,andreu2010nonlocal,burch2011classical,du2019cbms,du2013nonlocal,du2011mathematical,emmrich2007well}. Meanwhile, to simulate nonlocal models, various numerical approximations have been proposed and studied, including finite difference, finite element, meshfree method, quadrature and particle-based methods \cite{bobaru2009convergence,chen2011continuous,du2013posteriori,kilic2010coupling,macek2007peridynamics,seleson2009peridynamics,silling2005meshfree,tian2013analysis,tian2014asymptotically,wang2012fast,zhou2010mathematical}.

Let $\Omega\subset\mathbb{R}^{d}$ denote a bounded, open and convex domain with Lipschitz continuous boundary. For $u({\bf x}): \Omega\rightarrow \mathbb{R}$, the nonlocal operator $\mathcal{L}$ on $u({\bf x})$ is defined as
\begin{equation}\label{Nolocal_Operator}
\mathcal{L}u({\bf x})=\int_{\mathbb{R}^{d}}{\big (}u({\bf y})-u({\bf x}){\big )}\gamma({\bf x},{\bf y})d{\bf y}\quad \forall {\bf x}\in\Omega,
\end{equation}
where the nonnegative symmetric mapping $\gamma({\bf x},{\bf y}): \mathbb{R}^{d}\times\mathbb{R}^{d} \rightarrow \mathbb{R}$ is called a kernel. The operator $\mathcal{L}$ is regarded nonlocal since the value of $\mathcal{L}u$ at a point $\bf x$ involves information about $u$ at points ${\bf y}\neq{\bf x}$. In this article, we consider the following nonlocal Dirichlet volume-constrained diffusion problem
\begin{equation}\label{nonlocal_diffusion}
     \left \{
     \begin{array}{rl}
-\mathcal{L}u({\bf x})&=b({\bf x})\: \mbox{on}\: \Omega, \\
u({\bf x})&=g({\bf x}) \: \mbox{on} \:\Omega_I,
     \end{array}
     \right .
\end{equation}
where $\Omega_I=\{{\bf y}\in \mathbb{R}^d\setminus\Omega:\, \mbox{dist}({\bf y},\partial \Omega)<\delta\}$ with $\delta$ a constant called horizon parameter, $b({\bf x})\in L^2(\Omega)$ and $g({\bf x})\in L^2(\Omega_I)$ are given functions. Volume constraints are natural extensions, to the nonlocal case, of
boundary conditions for differential equations. Nonlocal versions of Neumann and Robin boundary conditions are also naturally defined \cite{du2012analysis}.

We assume that the nonnegative symmetric kernel $\gamma({\bf x},{\bf y})$
satisfies, there exists a positive constant $\gamma_0$, for all ${\bf x}\in\Omega\cup\Omega_I$,
\begin{equation}\label{kernel_basic}
\gamma({\bf x},{\bf y})\geq \gamma_0\quad\forall {\bf y}\in B_{\delta/2}({\bf x}),\quad
\gamma({\bf x},{\bf y})=0\quad \forall {\bf y}\in(\Omega\cup\Omega_I)\setminus B_{\delta}({\bf x}),
\end{equation}
where $B_{\delta}({\bf x}) := \{{\bf y}\in \Omega\cup\Omega_I: |{\bf y}-{\bf x}|\leq\delta\}$. Obviously, (\ref{kernel_basic}) implies that although
interactions are nonlocal, they are limited to a ball of radius $\delta$.
A few important classes of kernels are considered in  \cite{du2012analysis}. Of particular interests here
is  a special choice of $\gamma$ being a radial function of
${\bf x}-{\bf y}$ (which also makes the kernel translation invariant).
Such a case has also been studied earlier in \cite{andreu2010nonlocal} where
\begin{equation}\label{int_kernel}
\gamma({\bf x},{\bf y})=\tilde{\gamma}(|{\bf y}-{\bf x}|)\geq 0, \quad
\int_{\mathbb{R}^{d}}\tilde{\gamma}(|{\bf z}|) d{\bf z} = 1.
\end{equation}
This condition on $\gamma$
implies that $\mathcal{L}$ is a bounded mapping from $L^2(\mathbb{R}^{d})$ to $L^2(\mathbb{R}^{d})$. As we will discuss here, even though for smooth enough $b({\bf x})$ in $\Omega$, unlike the classical local PDE boundary value problems, the solution $u({\bf x})$ is possibly discontinuous across $\partial\Omega$ which makes the numerical solution to the corresponding nonlocal problem challenging.

An intuitive idea to overcome the possible loss of continuity is to use  discontinuous Galerkin (DG) methods. The latter are in fact conforming, which is in stark contrast to DG methods for the discretization of second order elliptic partial differential equations for which they are nonconforming \cite{arnold2002unified}.
While nonconforming DG has also been studied recently for nonlocal models \cite{du18dg},
if the structure of the solution could be studied carefully, a well designed conforming DG method
could be a more competitive option as it could lessen  the cost of computation. In this work,
we propose a new kind of conforming DG method to approximate the nonlocal problem (\ref{nonlocal_diffusion}) with kernels satisfying (\ref{kernel_basic}) and (\ref{int_kernel}) where the
key is to adopt a hybrid version of continuous elements with DG at the boundary.

The paper is organized as follows. In Section \ref{section:structure}, the structure of the solution for the given problem is analyzed, which is a generalization of the results in \cite{silling2003deformation}. We also convert the original inhomogeneous problem (\ref{nonlocal_diffusion}) into the homogeneous problem (\ref{modified_problem}) whose solution is smoother, so we just need to discuss a smoother homogeneous problem (\ref{nonlocal_diff_homo}). Based on the structure of the solution, in Section \ref{section:DG_method} we propose a new DG method which solves the problem (\ref{nonlocal_diff_homo}) efficiently. Convergence analysis and condition number estimates along with asymptotic compatibility  for the method are given in Section \ref{section:Theoretical considerations}. Results of numerical experiments are reported in Section \ref{section:numerical_experiment}.

\section{The structure of the solution}\label{section:structure}
To design more efficient numerical discretization, we first present some regularity studies on the nonlocal constrained value problem.  We recall first some one dimensional results presented in \cite{silling2003deformation}: when using peridynamic theory to model the elasticity on $\mathbb{R}=(-\infty,\infty)$, the displacement field $u$ has the same smoothness as the body force field $b$. In addition, any discontinuity in the kernel $\gamma$ (or in one of its derivatives) has some additional effect on the smoothness of $u$. For a peridynamic bar, suppose that $b$ has a discontinuity in its $N$th derivative at some $x=x_b$, and $\gamma$ has a discontinuity in its $L$th derivative at some $x=x_c$, then $u$ has a discontinuity in its $(N + nL + n)$th derivative at $x=x_b+nx_c, n=1,2,\cdots$. This propagation of discontinuities is illustrated schematically in \cite[Figure 3]{silling2003deformation}. Roughly speaking, the smoothness of $u$ increases as one moves away from
the location where the solution is discontinuous due to the discontinuity of the body force field $b$.
These types of step-wise improved regularity associated with a finite horizon parameter have also been observed for nonlocal initial value problems of nonlocal-in-time dynamic systems in \cite{du17dcdsb}.

\subsection{The structure of the solution for general dimensional cases on bounded domains}
Recall for the 1D case in an unbounded domain, the regularity of the solution for a nonlocal problem is affected by both the right hand side function and the kernel function, assuming good behavior of the solution at infinities. In this subsection we consider the effect of these two sources on the regularity of the solution for general multidimensional constrained value problem on a bounded domain.

First, let us present a result to reduce  the problem (\ref{nonlocal_diffusion}) which we are concerned with to be a problem with a homogeneous nonlocal constraint.
Denoted by
\begin{equation}\label{btrans}
     \overline{b}({\bf x})=\left \{
     \begin{array}{ll}
     b({\bf x}), & {\bf x}\in \Omega,\\
     g({\bf x}), & {\bf x}\in \Omega_I,
     \end{array}
     \right .
\end{equation}
and
\begin{equation}\label{utrans}
\overline{u}({\bf x})=u({\bf x})-\overline{b}({\bf x}).
\end{equation}
Then the nonlocal operator $-\mathcal{L}$ on $\overline{u}({\bf x})$ is
\begin{align*}
-\mathcal{L}\overline{u}({\bf x})&=f({\bf x})=\int_{B_{\delta}({\bf x})\cap\Omega}\overline{b}({\bf y})\gamma({\bf y}-{\bf x})d{\bf y}\\
&=\int_{B_{\delta}({\bf x})\cap\Omega}b({\bf y})\gamma({\bf y}-{\bf x})d{\bf y}+\int_{B_{\delta}({\bf x})\cap\Omega_I}g({\bf y})\gamma({\bf y}-{\bf x})d{\bf y}, \quad \forall {\bf x} \in \Omega.
\end{align*}
Thus, $\overline{u}({\bf x})$ is the solution of the following homogeneous nonlocal problem
\begin{equation}\label{modified_problem}
     \left \{
     \begin{array}{rl}
-\mathcal{L}\overline{u}({\bf x})&=f({\bf x}),\:\: \mbox{on}\: \Omega, \\
\overline{u}({\bf x})&=0,\qquad   \mbox{on} \:\Omega_I.
     \end{array}
     \right .
\end{equation}
Due to $\overline{b}({\bf y})\in L^2(\Omega\cup\Omega_I)$, $\gamma({\bf s})\in L^2(\mathbb{R}^d)$, and the fact that convolution of functions in dual $L^p(\mathbb{R}^d)$-spaces is continuous, we know that $f({\bf x})\in C(\Omega)$.

The problem (\ref{nonlocal_diffusion}) is equivalent to the problem (\ref{modified_problem}) which has a homogeneous nonlocal constraint. That is, we just need to study the following homogeneous nonlocal problem:
\begin{equation}\label{nonlocal_diff_homo}
     \left \{
     \begin{array}{rl}
-\mathcal{L}u({\bf x})&=b({\bf x}),\:\: \mbox{on}\: \Omega, \\
u({\bf x})&=0,\quad\:\:\:   \mbox{on} \:\Omega_I,
     \end{array}
     \right .
\end{equation}
with $b({\bf x})\in C(\Omega)$, $\gamma({\bf x},{\bf y})$ satisfying (\ref{kernel_basic}) and (\ref{int_kernel}).
We will show that if $\gamma({\bf x},{\bf y})$ satisfies some mild assumptions, the results recalled earlier for the one dimensional case can be generalized to multidimensional case on a bounded domain.

\subsection{Discontinuities due to the right hand side function}
\begin{theorem}\label{Thm:Continuity_Zeroboundary}
If $\gamma({\bf x},{\bf y})$ satisfies (\ref{kernel_basic}) and (\ref{int_kernel}), and $b({\bf x})\in C(\Omega)$, the solution of (\ref{nonlocal_diff_homo})
is continuous in $\Omega$, i.e., $u({\bf x})\in C(\Omega)$.
\end{theorem}
\begin{proof}
Since $\gamma({\bf x},{\bf y})$ satisfies (\ref{kernel_basic}) and (\ref{int_kernel}), we easily see that
\begin{align*}
u({\bf x})\in L^2(\Omega\cup\Omega_I).
\end{align*}
From (\ref{nonlocal_diff_homo}), we have
\begin{align}\label{u_relation}
u({\bf x})=b({\bf x})+\int_{B_{\delta}({\bf x})}u({\bf y})\gamma({\bf y}-{\bf x})d{\bf y}, \quad \forall {\bf x} \in \Omega.
\end{align}
Since $u({\bf y})\in L^2(\Omega\cup\Omega_I)$ and $\gamma({\bf s})\in L^2(\mathbb{R}^d)$,
we have that
$$\int_{B_{\delta}({\bf x})}u({\bf y})\gamma({\bf y}-{\bf x})d{\bf y}$$ is continuous for ${\bf x}$ in $\Omega$. Together with the condition $b({\bf x})\in C(\Omega)$, we complete the proof. 
\end{proof}

It is obvious from (\ref{u_relation}) that the discontinuity at a point $\bf x$ of the right hand side $b$ will lead to the discontinuity at the same point of the solution $u$. With an additional assumption on the kernel, we can bootstrap a higher order regularity result as follows.

\begin{theorem}\label{Thm:H1:smooth}
Suppose that $\gamma({\bf x},{\bf y})$ satisfies (\ref{kernel_basic}) and (\ref{int_kernel}). If $\tilde{\gamma}(r)\in C^1(0,\delta)$, $\tilde{\gamma}(\delta)=0$ and $b({\bf x})\in C^1(\Omega)$, then $u({\bf x})\in C^1(\Omega)$.
\end{theorem}
\begin{proof}
From (\ref{u_relation}), we know
\begin{align}\label{u_relation_s}
u({\bf x})=b({\bf x})+\int_{B_{\delta}({\bf 0})}u({\bf x}+{\bf s})\gamma({\bf s})d{\bf s}.
\end{align}
For any unit vector ${\bf t}$, we take the corresponding directional derivative for  (\ref{u_relation_s}), so
\begin{align*}
\frac{\partial u({\bf x})}{\partial{\bf t}}=\frac{\partial b({\bf x})}{\partial{\bf t}}+\int_{B_{\delta}({\bf 0})}\frac{\partial u({\bf x}+{\bf s})}{\partial{\bf t}}\gamma({\bf s})d{\bf s}.
\end{align*}
Since $\tilde{\gamma}(r)\in C^1(0,\delta)$ and $\tilde{\gamma}(\delta)=0$, we have
\begin{align*}
&\int_{B_{\delta}({\bf 0})}\frac{\partial u({\bf x}+{\bf s})}{\partial{\bf t}}\gamma({\bf s})d{\bf s}\\
&=\int_{\partial B_{\delta}({\bf 0})}u({\bf x}+{\bf s})\gamma({\bf s}){\bf t}\cdot{\bf n}_{\bf s}d{\bf S}_{\bf s}-\int_{B_{\delta}({\bf 0})}u({\bf x}+{\bf s}){\bf t}\cdot\nabla\gamma({\bf s})d{\bf s}\\
&=\tilde{\gamma}(\delta)\int_{\partial B_{\delta}({\bf 0})}u({\bf x}+{\bf s}){\bf t}\cdot{\bf n}_{\bf s}d{\bf S}_{\bf s}-\int_{B_{\delta}({\bf 0})}u({\bf x}+{\bf s})\tilde{\gamma}'(|{\bf s}|){\bf t}\cdot{\bf n}_{\bf s}d{\bf s}\\
&=-\int_{B_{\delta}({\bf 0})}u({\bf x}+{\bf s})\tilde{\gamma}'(|{\bf s}|){\bf t}\cdot{\bf n}_{\bf s}d{\bf s}.
\end{align*}
Thus,
\begin{align}\label{u_de_intbypart}
\frac{\partial u({\bf x})}{\partial{\bf t}}=\frac{\partial b({\bf x})}{\partial{\bf t}}-\int_{B_{\delta}({\bf 0})}u({\bf x}+{\bf s})\tilde{\gamma}'(|{\bf s}|){\bf t}\cdot{\bf n}_{\bf s}d{\bf s}.
\end{align}
Since the convolution of functions in dual $L^p(\mathbb{R}^{d})$-spaces is continuous, the second term in the right hand side of (\ref{u_de_intbypart}) is continuous with respect to ${\bf x}$. Together with the condition $b({\bf x})\in C^1(\Omega)$, we complete the proof. 
\end{proof}

From (\ref{u_de_intbypart}) we get that under the assumptions of Theorem \ref{Thm:H1:smooth},
if the first derivative of the right hand side $b$ is discontinuous at some point $\bf x$, the first derivative of the solution $u$ will be discontinuous at the same point. However, if we may have both the condition $\tilde{\gamma}(\delta)>0$ and the condition $b({\bf x})\notin C^1(\Omega)$, the conclusion $u({\bf x})\in C^1(\Omega)$ may still hold (see Example 1, where in fact $u({\bf x})\in C^{\infty}(\Omega)$). This
is not contradicting as \emph{the sum of two discontinuous function could be continuous, and even infinitely differentiable}.

\subsection{Propagation of discontinuities due to the kernel}
For the convenience of discussion, let us denote by $\Omega_{1}=\{{\bf x}\in \Omega: \mbox{dist}(x,\partial\Omega)>\delta\}$ and $\Omega_{2}=\Omega\setminus\overline{\Omega_{1}}$. In fact, the
significance of the
result Theorem \ref{Thm:Continuity_Zeroboundary} is twofold. First, it indicates \emph{the smoothness of $u({\bf x})$ is the same as $b({\bf x})$ for general multidimensional case on a bounded domain}. Second, it reveals the possible propagation  of discontinuities due to the kernel. Although $b({\bf x})\in C(\Omega)$, it might be discontinuous across $\partial\Omega$. So does $u({\bf x})$, and this discontinuity will be propagated to those points on $\partial\Omega_{1}$, which are $\delta$ distance from $\partial\Omega$ onto one order higher derivatives  by Theorem \ref{Thm:Continuity_Zeroboundary}, which is consistent with the conclusion for 1D case on $\mathbb{R}$ with $N=0$ and $L=0$. Similar argument can be given for Theorem \ref{Thm:H1:smooth} which is consistent with the conclusion for 1D case on $\mathbb{R}$ with $N=0$ and $L=1$. This bootstrap procedure could be repeated again and again, and the corresponding results for general $N$ and $L$ then follow.

Let us now emphasize  on the importance and necessity for the smoothness of the kernel function. For instance, in Theorem \ref{Thm:H1:smooth} $\tilde{\gamma}(\delta)=0$ is required such that $\tilde{\gamma}$ has a discontinuity in its first (but not zeroth) derivative at $x=\delta$. If $\tilde{\gamma}(\delta)>0$, then from the proof of the Theorem \ref{Thm:H1:smooth}, we see that  $u({\bf x})\in C^1(\Omega)$ may not hold. In fact, for all ${\bf x}_0\in\partial\Omega_{1}$ and any unit vector ${\bf t}$, since
\begin{align*}
\frac{\partial u({\bf x}_0)}{\partial{\bf t}}&=\frac{\partial b({\bf x}_0)}{\partial{\bf t}}+\tilde{\gamma}(\delta)\int_{\partial B_{\delta}({\bf 0})}u({\bf x}_0+{\bf s}){\bf t}\cdot{\bf n}_{\bf s}d{\bf S}_{\bf s}\\
&-\int_{B_{\delta}({\bf 0})}u({\bf x}_0+{\bf s})\tilde{\gamma}'(|{\bf s}|){\bf t}\cdot{\bf n}_{\bf s}d{\bf s},
\end{align*}
if $\tilde{\gamma}(\delta)>0$ (that is $\tilde{\gamma}$ has a discontinuity in its zeroth derivative) and $b$ (thus $u$) is discontinuous across $\partial\Omega$, then the term
\begin{equation*}
\int_{\partial B_{\delta}({\bf 0})}u({\bf x}+{\bf s}){\bf t}\cdot{\bf n}_{\bf s}d{\bf S}_{\bf s}
\end{equation*}
is likely to be discontinuous across ${\bf x}_0$. If so, $\frac{\partial u}{\partial{\bf t}}$ would be discontinuous at ${\bf x}_0$ ($u$ has a discontinuity in its first derivative). This situation might happen, as illustrated in Example 2. Using a similar argument we could prove the following theorem.

\begin{theorem}\label{Theorem:C2}
Suppose that $\gamma({\bf x},{\bf y})$ satisfies (\ref{kernel_basic}) and (\ref{int_kernel}). If the following two conditions hold:

(i) $b({\bf x})\in C^1(\Omega)$, $b\in C^{2}(\Omega_{1})$, $b\in C^2(\Omega_{2})$,

(ii) $\tilde{\gamma}(r)\in C^1(0,\delta)$ and $\tilde{\gamma}(\delta)=0$.\\
Then $u\in C^1(\Omega)$, $u\in C^2(\Omega_{1})$ and $u\in C^2(\Omega_{2})$.
\end{theorem}
\begin{proof}
The conditions of Theorem \ref{Thm:H1:smooth} hold due to the conditions (i) and (ii), so $u\in C^{1}(\Omega)$. Furthermore, for
any two unit vectors ${\bf t}_1$ and ${\bf t}_2$, take a directional derivative for (\ref{u_de_intbypart}), 
\begin{align*}
\frac{\partial^2 u({\bf x})}{\partial{\bf t}_1\partial{\bf t}_2}&=\frac{\partial^2 b({\bf x})}{\partial{\bf t}_1\partial{\bf t}_2}
-\int_{B_{\delta}({\bf 0})}\frac{\partial u({\bf x}+{\bf s})}{\partial {\bf t}_2}\tilde{\gamma}'(|{\bf s}|){\bf t}_1\cdot{\bf n}_{\bf s}d{\bf s}.
\end{align*}
Applying this equality in $\Omega_{1}$ and $\Omega_{2}$ will lead to the conclusion $u\in C^{2}(\Omega_{1})$ and $u\in C^{2}(\Omega_{2})$, respectively. 
\end{proof}
To get the optimal convergence order, we always need the regularity $u\in H^2(\Omega)$, the following corollary give a sufficient condition to guarantee this property.
\begin{corollary}\label{Corollary:H2}
Suppose that $\gamma({\bf x},{\bf y})$ satisfies (\ref{kernel_basic}) and (\ref{int_kernel}). If the following two conditions hold:

(i) $b({\bf x})\in C^1(\Omega)$, $b\in H^{2}(\Omega_{1})$, $b\in H^2(\Omega_{2})$,

(ii) $\tilde{\gamma}(r)\in C^1(0,\delta)$ and $\tilde{\gamma}(\delta)=0$.\\
Then $u\in H^2(\Omega)$.
\end{corollary}
\begin{proof}
Using the density of $C^2(\Omega_{i})$ in $H^2(\Omega_{i})$ ($i=1,2$) and Theorem \ref{Theorem:C2} we could get the result $u\in H^{2}(\Omega_{i})$. Since $u\in C^{1}(\Omega)$ is proven, the result $u\in H^{2}(\Omega)$ holds. 
\end{proof}

\section{A new DG method for nonlocal models with integrable kernels}\label{section:DG_method}
Here and after, for a function $\phi({\bf x})$ we denote $\lim\limits_{h\rightarrow 0-}\phi({\bf x}+h{\bf n_x})$ by $\phi({\bf x}-)$ in the case of no ambiguity.
Under the condition of Theorem \ref{Thm:Continuity_Zeroboundary}, we know that for given ${\bf x}\in \partial \Omega$,
\begin{equation*}
\lim\limits_{h\rightarrow 0+}u({\bf x}+h{\bf t})=u({\bf x}-), \forall {\bf t}\cdot{\bf n}_{\bf x}<0.
\end{equation*}
However $u({\bf x}-)$ does not need to be zero, that is $u({\bf x})$ is possibly discontinuous across $\partial\Omega$. Thus, it might be inefficient to use continuous FEM on the whole domain $\Omega\cup\Omega_I$. Moreover, since we do not specify the value of the right hand side function on $\Omega_I$ a priori, we have no control on the amount of the jump across $\partial\Omega$ where the solution is likely to be discontinuous.  Thus,
we propose a suitable DG method by adopting a hybrid version of DG (across the domain boundary) and continuous elements (in the interior domain).

As in \cite{du2012analysis} the nonlocal energy inner product, the nonlocal energy norm, nonlocal energy space, and nonlocal volume constrained energy space are defined by
\begin{equation*}
(u,v)_{\||} :={\Big (}\int_{\Omega\cup\Omega_I}\int_{\Omega\cup\Omega_I}{\big(}u({\bf y})-u({\bf x}){\big)}{\big(}v({\bf y})-v({\bf x}){\big)}\gamma({\bf x},{\bf y})d{\bf y}d{\bf x}{\Big )},
\end{equation*}
\begin{equation*}
\|| u\|| :=(u,u)^{1/2}_{\||},
\end{equation*}
\begin{equation*}
V(\Omega\cup\Omega_I):=\{u\in L^2(\Omega\cup\Omega_I): \|| u\||<\infty\},
\end{equation*}
\begin{equation*}
V_{c,0}(\Omega\cup\Omega_I):=\{u\in V(\Omega\cup\Omega_I): u({\bf x})=0\:\mbox{on} \:\Omega_I\},
\end{equation*}
respectively. Similar to the definition $V_{c,0}(\Omega\cup\Omega_I)$, the subspace of $L^{2}(\Omega\cup\Omega_I)$ is defined as follows:
\begin{equation*}
L^{2}_{c,0}(\Omega\cup\Omega_I):=\{u\in L^{2}(\Omega\cup\Omega_I): u({\bf x})=0\:\mbox{on} \:\Omega_I\}.
\end{equation*}
The authors in \cite{du2012analysis} prove that if the kernel function $\gamma({\bf x},{\bf y})$ satisfies (\ref{kernel_basic}) and (\ref{int_kernel}), then $V_{c,0}(\Omega\cup\Omega_I)$ is equivalent to $L^{2}_{c,0}(\Omega\cup\Omega_I)$. Denote by
\begin{equation*}
V_{c,0}^{pc}(\Omega\cup\Omega_I)=\{u\in V_{c,0}(\Omega\cup\Omega_I): u|_{\Omega} \in C(\Omega)\},
\end{equation*}
where the superscripts $pc$ means \emph{partly continuous} (continuous in $\Omega$).

For a given triangulation $\mathcal{T}_h$ of $\Omega\cup\Omega_I$ that simultaneously triangulates $\Omega$, let $\Omega_h=\mathcal{T}_h\cap\overline{\Omega}$.
Next, let $V^{pc,h}_{c,0}$ consist of those functions in $V_{c,0}^{pc}(\Omega\cup\Omega_I)$ that are piecewise linear. Since $\Omega$ is convex, this \emph{conforming} property is satisfied, that is,
\begin{equation}\label{Set_conf}
V_{c,0}^{pc,h}\subset V_{c,0}^{pc}(\Omega\cup\Omega_I).
\end{equation}

We assume that $\mathcal{T}_h$
is shape-regular and quasi-uniform \cite{brenner2007mathematical} as $h\rightarrow 0$, where $h$ denotes the diameter of the largest element in $\mathcal{T}_h$. For a fixed $\mathcal{T}_h$, the set of inner nodes of $\Omega_h$, i.e., all nodes in $\Omega_h\setminus\partial\Omega$, is denoted by $N\!I=\{{\bf x}_j: j=1,2,\cdots, m\}$, with piecewise linear basis functions defined on $\mathcal{T}_h$ being $\phi_j({\bf x}),\:j=1,2,\cdots, m$. The set of all nodes in $\Omega_h\cap\partial\Omega$ is denoted by $N\!B=\{{\bf x}_{m+j}: j=1,2,\cdots, n\}$ with piecewise linear basis functions defined on $\mathcal{T}_h$ being $\phi_{m+j}({\bf x}),\:j=1,2,\cdots, n$. The basis functions for the space $V_{c,0}^{pc,h}$ are as follows: for $j=1,2,\cdots,m+n$,
\begin{equation*}
     \widehat{\phi}_{j}({\bf x})=\left \{
     \begin{array}{ll}
     \phi_{j}({\bf x})|_{\Omega_h}, & {\bf x}\in \Omega_h,\\
     0, & {\bf x}\in (\Omega\cup\Omega_I)\setminus\Omega_{h}.
     \end{array}
     \right .
\end{equation*}
Throughout the paper, the generic constant $C$ is always independent of the finite element mesh parameter $h$.

Since we know the structure of the true solution and the space it belongs to, we could design a DG method for its approximation. First, variational form in $V_{c,0}^{pc}(\Omega\cup\Omega_I)$ finds $u({\bf x})\in V_{c,0}^{pc}(\Omega\cup\Omega_I)$, such that
\begin{equation}\label{modi_prob_var}
     -\int_{\Omega}\mathcal{L}u({\bf x}) w({\bf x})d{\bf x}=\int_{\Omega}b({\bf x})w({\bf x})d{\bf x},\: \forall w({\bf x}) \in V_{c,0}^{pc}(\Omega\cup\Omega_I).
\end{equation}
The finite dimensional approximation for (\ref{modi_prob_var}) finds $u_h({\bf x})\in V^{pc,h}_{c,0}$, such that
\begin{equation}\label{modi_disc_var}
-\int_{\Omega_{h}}\mathcal{L}u_h({\bf x}) w_h({\bf x})d{\bf x}=\int_{\Omega_{h}}b({\bf x})w_h({\bf x})d{\bf x}, \: \forall w_h({\bf x}) \in V_{c,0}^{pc,h}.
\end{equation}
Set $u_h({\bf x})=\sum\limits_{j=1}^{m+n}u_j\widehat{\phi}_{j}({\bf x})$, ${\bf u}=(u_{1}, u_{2},\cdots,u_{m+n})^T$. Denote by
\begin{align*}
{\bf d}=(d_1, d_2, \cdots, d_{m+n})^T,
\end{align*}
and
\begin{align*}
A_{II}=(a_{i,j})_{m\times m},\: A_{IB}=(a_{i,{m+j}})_{m\times n},\: A_{BB}=(a_{m+i,{m+j}})_{n\times n},
\end{align*}
with
\begin{align*}
d_i=\int_{\Omega_{h}}b({\bf x}) \widehat{\phi}_i({\bf x})d{\bf x},\:
a_{i,j}=-\int_{\Omega_{h}}\mathcal{L}\widehat{\phi}_j({\bf x}) \widehat{\phi}_i({\bf x})d{\bf x}.
\end{align*}
Set $w_h=\widehat{\phi}_i$, $i=1,2,\cdots,m+n$, the algebraic system of (\ref{modi_disc_var}) is
\begin{align*}
A{\bf u}={\bf d},
\end{align*}
with
\begin{equation}\label{Stiff_matrix}
A=\left(\begin{array}{cc}
A_{II} &  A_{IB} \\
A_{IB}^T& A_{BB}
\end{array}\right).
\end{equation}

In the process to solve for $u_h({\bf x})$ we use the finite element space $V^{pc,h}_{c,0}$ which is continuous in $\Omega_{h}$, however, discontinuous across $\partial\Omega_{h}$, thus we regard it a
conforming but hybrid version of DG and
continuous FEM. This method possesses some advantages as follows:

(i) The method leads to a linear algebraic system with the coefficient matrix $A$ in (\ref{Stiff_matrix}) that is symmetric and positive definite,
just as in the case using either the conforming DG or continuous FEM,
thus many efficient solvers suitable to such linear systems could still be used.

(ii) The method is asymptotically compatible: as shown in \cite{tian2014asymptotically}, as long as the finite element space
contains continuous piecewise linear functions (which is the case for our hybrid algorithm), the Galerkin finite element approximation is always
asymptotically compatible, and thus offers robust numerical discretizations to problems involving nonlocal interactions.

(iii) The method has optimal convergence rate provided that the solution is smooth on $\Omega$, that is $O(h^2)$ ($O(h)$) for error in $L^2$ ($H^1$) norm provided that the true solution $u\in H^2(\Omega)$. This result is in sharp contrast to the assumption given that in \cite{du2012analysis} where to insure the optimal convergence rate the true solution is required to be in $H^2(\Omega\cup\Omega_I)$ which generally not holds for the problem (\ref{nonlocal_diff_homo}).  Furthermore, it has optimal convergence rate for 1D case under very weak assumptions, and nearly optimal convergence rate for two dimensional (2D) case under some mild assumptions, as shown in the next section.

(iv) The method, in comparison with the direct use of DG in all discrete elements,  uses a smaller degree of freedoms. For example, the degree of freedoms is $n+1$ versus $2n$ for a mesh with $n+1$  nodes in 1D case, and $(n+1)^2$ versus $6n^2$ for  a uniform triangulation with $n^2$ nodes in 2D case.

\section{Theoretical Analysis}\label{section:Theoretical considerations}
We now provide further theoretical analysis on the new DG approximations. Given what has already
been discussed in (ii) of the above section,  the asymptotic compatibility is assured and we thus
focus on the case where the problems remain strictly nonlocal.

\subsection{Convergence}
The following convergence result describes the best approximation property of the finite-dimensional Ritz-Galerkin solution.
\begin{theorem}\label{Thm:Convergence}
If $\gamma({\bf x},{\bf y})$ satisfies (\ref{kernel_basic}) and (\ref{int_kernel}), $b({\bf x})\in C(\Omega)$, $u({\bf x})$ is the solution of (\ref{nonlocal_diff_homo}), $u_{h}({\bf x})$ is the solution of (\ref{modi_disc_var}). We define
\begin{equation*}
     \tilde{u}({\bf x},\Omega_h)=\left \{
     \begin{array}{ll}
     u({\bf x}), & {\bf x}\in \Omega_h,\\
     0, & {\bf x}\in (\Omega\cup\Omega_I)\setminus\Omega_{h}.
     \end{array}
     \right .
\end{equation*}
Then we have
\begin{equation}\label{cea0}
\|| \tilde{u}-u_{h}\||\leq \inf_{w_h \in V_{c,0}^{pc,h}}
\|| \tilde{u}-w_h\||\, .
\end{equation}
Consequently,
\begin{equation}\label{cea}
\|u-u_{h}\|_{\Omega_{h}}\leq C \min_{w_{h}\in V^{pc,h}_{c,0}}\|u-w_{h}\|_{\Omega_{h}}\rightarrow 0 \quad as \: h\rightarrow 0.
\end{equation}
\end{theorem}

\begin{proof}
Since $V^{pc,h}_{c,0}\subset V^{pc}_{c,0}(\Omega\cup\Omega_I)$ as in (\ref{Set_conf}), then for all $w_h \in V_{c,0}^{pc,h}$,
\begin{equation*}
-\int_{\Omega_{h}}\mathcal{L}\tilde{u}({\bf x},\Omega_h) w_h({\bf x})d{\bf x}=\int_{\Omega_{h}}b({\bf x})w_h({\bf x})d{\bf x},
\end{equation*}
together with (\ref{modi_disc_var}), we have
\begin{equation*}
-\int_{\Omega_{h}}\mathcal{L}{\big(}\tilde{u}({\bf x},\Omega_h)-u_{h}({\bf x}){\big)} w_h({\bf x})d{\bf x}=0, \: \forall w_h \in V_{c,0}^{pc,h}.
\end{equation*}
Using the nonlocal Green's first identity \cite{du2013nonlocal}, we have
\begin{equation*}
(\tilde{u}-u_{h},w_h)_{\||}=0,\:\forall w_h \in V_{c,0}^{pc,h}.
\end{equation*}
Then we get the following estimate
\begin{align*}
\|| \tilde{u}-u_{h}\||^2 &=(\tilde{u}-u_{h},\tilde{u}-u_{h})_{\||}=(\tilde{u}-u_{h},\tilde{u}-w_h)_{\||}\\
&\leq \|| \tilde{u}-u_{h}\|| \|| \tilde{u}-w_h\||, \: \forall w_h({\bf x}) \in V_{c,0}^{pc,h},
\end{align*}
and then
\begin{align*}
\|| \tilde{u}-u_{h}\||\leq\|| \tilde{u}-w_h\||, \: \forall w_h({\bf x}) \in V_{c,0}^{pc,h}.
\end{align*}
By the equivalence between $V_{c,0}(\Omega\cup\Omega_I)$ and $L^{2}_{c,0}(\Omega\cup\Omega_I)$, we complete the proof.
\end{proof}

Let us note that due to the use of norm  equivalence in the above proof, generally speaking, the constant $C$ in the lemma could depend on the nonlocal space and thus the nonlocal kernel. One may not infer that this constant remains  uniformly bounded when we consider the local limit of the nonlocal problem. Fortunately, as
alluded to earlier, with the asymptotic compatibility already established in \cite{tian2014asymptotically},
we hereby only focus on the strict nonlocal case.

We now combine the theory of the interpolation error estimate and (\ref{cea}) to give the convergence rate estimate.
\begin{theorem}\label{Thm:Convergence_rate}
If $\gamma({\bf x},{\bf y})$ satisfies (\ref{kernel_basic}) and (\ref{int_kernel}), $b({\bf x})\in C(\Omega)$, $u({\bf x})$ is the solution of (\ref{nonlocal_diff_homo}), $u_{h}({\bf x})$ is the solution of (\ref{modi_disc_var}).
Suppose that $u\in H^t(\Omega)$ holds, there exists a constant $C$ such that, for sufficiently small $h$,
\begin{equation}\label{Error_u}
\|u-u_{h}\|_{\Omega_{h}}\leq C h^s\|u\|_{s,\Omega},
\end{equation}
with $s=\min(t,2)>d/2$. If $s>1$
\begin{equation}\label{Error_du}
\|\nabla(u-u_{h})\|_{\Omega_{h}}\leq C h^{s-1}\|u\|_{s,\Omega}.
\end{equation}
Moreover, if the following two conditions hold:

(i) $b({\bf x})\in C^1(\Omega)$, $b\in H^{2}(\Omega_{1})$, $b\in H^2(\Omega_{2})$;

(ii) $\gamma({\bf x},{\bf y})$ is a radial function such that $\tilde{\gamma}(r)\in C^1(0,\delta)$ and $\tilde{\gamma}(\delta)=0$.\\
Then $u\in H^2(\Omega)$, thus $s=2$.
\end{theorem}
\begin{proof}
Denote by $\mathcal{I}_hu$ the Lagrange interpolant from $C(\Omega_h)$ to $V_{c,0}^{pc,h}|_{\Omega_h}$\begin{equation*}
     w_{h}({\bf x})=\left \{
     \begin{array}{ll}
     \mathcal{I}_hu({\bf x}), & {\bf x}\in \Omega_h,\\
     0, & {\bf x}\in (\Omega\cup\Omega_I)\setminus\Omega_{h},
     \end{array}
     \right .
\end{equation*}
then $w_{h}\in V^{pc,h}_{c,0}$, and
\begin{equation}\label{Interp_err}
\|u-w_{h}\|_{\Omega_{h}}=\|u-\mathcal{I}_hu\|_{\Omega_{h}}\leq C h^{s}\|u\|_{s,\Omega},
\end{equation}
with $s=\min(t,2)$.
Combination of (\ref{cea}) and (\ref{Interp_err}) leads to (\ref{Error_u}).

Using the inverse estimate for finite element space, we have
\begin{align*}
\|\nabla(u-u_{h})\|_{\Omega_{h}}
&\leq\|\nabla(\mathcal{I}_hu-u_{h})\|_{\Omega_{h}}+\|\nabla(u-\mathcal{I}_hu)\|_{\Omega_{h}}\\
&\leq Ch^{-1}\|\mathcal{I}_hu-u_{h}\|_{\Omega_{h}}+ C h^{s-1}\|u\|_{s,\Omega}\leq C h^{s-1}\|u\|_{s,\Omega}
\end{align*}
This is the desired result (\ref{Error_du}).

The conditions (i) and (ii) lead to $u\in H^{2}(\Omega)$ due to Corollary \ref{Corollary:H2}. 
\end{proof}

We recall by Theorem 6.2 in \cite{du2012analysis} that, when continuous FEM is used to approximate the nonlocal problem (\ref{nonlocal_diff_homo}), the approximation $u_{n}$ has an error estimate of the form
$\|u-u_{n}\|_{\Omega\cup\Omega_I}\leq C h^s\|u\|_{s,\Omega\cup\Omega_I}$.
Since the solution of nonlocal problem (\ref{nonlocal_diff_homo}) could be discontinuous across $\partial\Omega$, we see that $u\in H^{s}(\Omega\cup\Omega_I)$ does not hold for $s\geq 1/2$, let alone for $s=2$. For 1D case, the best to expect is $s=1/2-\epsilon$ for arbitrary small positive $\epsilon$  if continuous FEM is used. Theorem \ref{Thm:Convergence_rate} improves the convergence rate from $1/2-\epsilon$ to $3/2-\epsilon$ for this case since we have the regularity of $H^{3/2-\epsilon}(\Omega)$.
This convergence rate is still not optimal. If the points on $\partial\Omega_{1}$ are selected as the mesh grids, the optimal convergence rate could be obtained.
\begin{theorem}
For the 1D case, assume $u\in C(0,1)$, $u\in H^2(0,\delta)$, $u\in H^2(\delta,1-\delta)$, $u\in H^2(1-\delta,1)$. If $\delta$ and $1-\delta$ are all selected as the mesh grids, then there exists a constant $C$ such that, for sufficiently small $h$,
\begin{equation}\label{Error_1D}
\|u-u_{h}\|_{(0,1)}+h\|u'-u'_{h}\|_{(0,1)}\leq C h^2{\big(}\|u\|_{2,(0,\delta)}+\|u\|_{2,(\delta,1-\delta)}+\|u\|_{2,(1-\delta,1)}{\big)}.
\end{equation}
\end{theorem}
\begin{proof}
Using the interpolation error estimate in three intervals $(0,\delta)$, $(\delta,1-\delta)$ and $(1-\delta,1)$, respectively and add them together, we get
\begin{equation*}
\|u-\mathcal{I}_hu\|_{(0,1)}\leq C h^2{\big(}\|u\|_{2,(0,\delta)}+\|u\|_{2,(\delta,1-\delta)}+\|u\|_{2,(1-\delta,1)}{\big)}.
\end{equation*}
Together with (\ref{cea}) and the inverse estimate we get (\ref{Error_1D}). 
\end{proof}

For 2D case, under mild assumptions on the smoothness of the boundary, the regularity of the solution, and the conformity between the mesh and the boundary, we have the almost optimal convergence rate, that is optimal up to a factor $|\log h|^{1/2}$.
\begin{theorem}
For 2D case, assume $u\in C(\Omega)$, $u\in H^{2}(\Omega_{1})$, $u\in H^2(\Omega_{2})$. If $\partial\Omega$ is of class $C^2$, and $N\!B\subset\partial \Omega_{1}$, then there exists a constant $C$ such that, for sufficiently small $h$,
\begin{equation*}
\|u-u_{h}\|_{\Omega_{h}}+h\|\nabla(u-u_{h})\|_{\Omega_{h}}\leq C h^2|\log h|^{1/2}{\big(}\|u\|_{2,\Omega_{1}}+\|u\|_{2,\Omega_{2}}{\big)}.
\end{equation*}
\end{theorem}
\begin{proof}
We cite from \cite{chen1998finite} the results for the linear interpolation error estimate for the interface problem, that is
\begin{equation*}
\|u-\mathcal{I}_hu\|_{\Omega_{h}}\leq C h^2|\log h|^{1/2}{\big(}\|u\|_{2,\Omega_{1}}+\|u\|_{2,\Omega_{2}}{\big)}.
\end{equation*}
Together with (\ref{cea}) and the inverse estimate we complete the proof. 
\end{proof}
\begin{remark}\label{Remark:kernel}
Due to the structure of the solution for the problem (\ref{nonlocal_diff_homo}) we have discussed, the solution $u$ is often discontinuous across $\partial\Omega$. It may cause the discontinuity of the first derivative across $\partial\Omega_{1}$ if $\tilde{\gamma}(r)$ has a discontinuity at $r=\delta$ (that is $\tilde{\gamma}(\delta)>0$), or the discontinuity of the second derivative across $\partial\Omega_{1}$ if $\tilde{\gamma}(r)$ is continuous at $r=\delta$ (that is $\tilde{\gamma}(\delta)=0$) but $\tilde{\gamma}'(r)$ has a discontinuity at $r=\delta$ (that is $\tilde{\gamma}'(\delta-)\neq 0$). In the next section, we will discuss two kinds of kernels representing the above two cases, respectively.
\end{remark}

\subsection{Condition number estimate}
The condition number of the stiffness matrix is an indicator of the sensitivity of the discrete solution with respect to the data and the performance of iterative solvers such as the conjugate-gradient method. For the DG method we propose in Section \ref{section:DG_method}, consider the $(m+n)\times(m+n)$ stiffness matrix $A$ defined in (\ref{Stiff_matrix}). We have the following condition number estimate whose proof is standard and is given for completeness.  Similar discussions can be found in earlier studies\cite{AU14,du10sinum}.
\begin{theorem}
For the stiffness matrix $A$ defined in (\ref{Stiff_matrix}) associated with the kernel $\gamma({\bf x},{\bf y})$ that satisfies (\ref{kernel_basic}) and (\ref{int_kernel}),
there exists a constant $C$ such that
\begin{equation*}
{\rm cond}_{2}(A)\leq C.
\end{equation*}
\end{theorem}
\begin{proof}
For the given finite element nodal basis, there exist two generic constants $c_{2}\geq c_{1}>0$ such that
\begin{equation*}
c_{1}h^{d}|{\bf w}|^{2}\leq \|w_{h}\|^{2}\leq c_{2}h^{d}|{\bf w}|^{2}\,,\quad \forall w_{h}=\sum_{j=1}^{m+n}w_{j}\widehat{\phi}_j\in V_{c,0}^{pc,h},
\end{equation*}
where $\{w_{j}\}$, $j=1,2,\cdots, m+n$, are components of the vector ${\bf w}$.
Since the space $V_{c,0}(\Omega\cup\Omega_I)$ is equivalent to the space $L^{2}_{c,0}(\Omega\cup\Omega_I)$, we get the theorem immediately. 
\end{proof}

\begin{remark}
We note again that the constant $C$ may depend on the kernel,  as demonstrated in \cite{du10sinum}, hence the result is only meaningful for nonlocal problems with a fixed kernel that satisfies the assumptions (\ref{kernel_basic}) and (\ref{int_kernel}).
\end{remark}

\section{Numerical results}\label{section:numerical_experiment}
We now report results of numerical experiments which substantiate the theoretical analysis in Section \ref{section:Theoretical considerations}. For 1D case, the problem (\ref{nonlocal_diff_homo}) becomes the following form
\begin{equation}\label{1D_prob}
     \left \{
     \begin{array}{rl}
-\int_{-\delta}^{\delta}{\big(}u(x+s)-u(x){\big)}\gamma(s)ds&=b(x)\quad \mbox{on}\: (0,1), \\
u(x)&=0 \: \qquad \mbox{on} \:(-\delta,0]\cup[1,1+\delta).
     \end{array}
     \right .
\end{equation}
We solve the nonlocal problem first on uniform meshes and take $\delta$ to be a constant multiple of $h$ and reduce $h$ to check convergence and condition number properties of the proposed DG method, and then solve the problem on non-uniform meshes which are obtained by random disturbance to uniform meshes. Here we choose two popular examples of kernel functions representing two cases as discussed in Remark \ref{Remark:kernel}.

\subsection{Constant kernel function}
We first consider the following kernel function
\begin{equation}\label{Con_ker}
     \gamma(s)=\left \{
     \begin{array}{ll}
     (2\delta)^{-1}, & |s|\leq \delta,\\
     0, & |s|>\delta.
     \end{array}
     \right .
\end{equation}
Obviously $\gamma$ defined as (\ref{Con_ker}) is discontinuous at points $\pm\delta$, if $b$ is discontinuous at $x=0$ or $x=1$, the solution of (\ref{1D_prob}) will probably be (however, not necessarily) discontinuous in its first derivative at $x=\delta$ or $x=1-\delta$. In fact, in Example 1 the smoothness pick-up is beyond first order, that is, although $b$ is discontinuous at $x=1$, $u$ is
infinitely continuously differentiable at $x=1-\delta$. While in Example 2 the smoothness pick-up is only first order and could not be improved, that is, since $b$ is discontinuous at $x=0$ and $x=1$, the first derivative of $u$ is
is discontinuous at $x=\delta$ and $x=1-\delta$.

\noindent{\bf Example 1.} In order to get simpler benchmark
solutions, we calculate the right-hand side of (\ref{1D_prob}) based on an exact solution $u(x)=x^2,\: x\in \Omega=(0,1)$ and $u(x)=0,\: x\in \Omega_I=(-\delta,0)\cup(1,1+\delta)$, with kernel function (\ref{Con_ker}). This naturally leads to a $\delta$-dependent right-hand side $b(x)=b_{\delta}(x)$, see Figure \ref{fig:1} for the plots of $u(x)$ and $b(x)$. The DG method we proposed in Section \ref{section:DG_method} is used to discretize it with $\delta=0.4$.
\begin{figure}
\subfloat[$b(x)$]{\includegraphics[width=6cm]{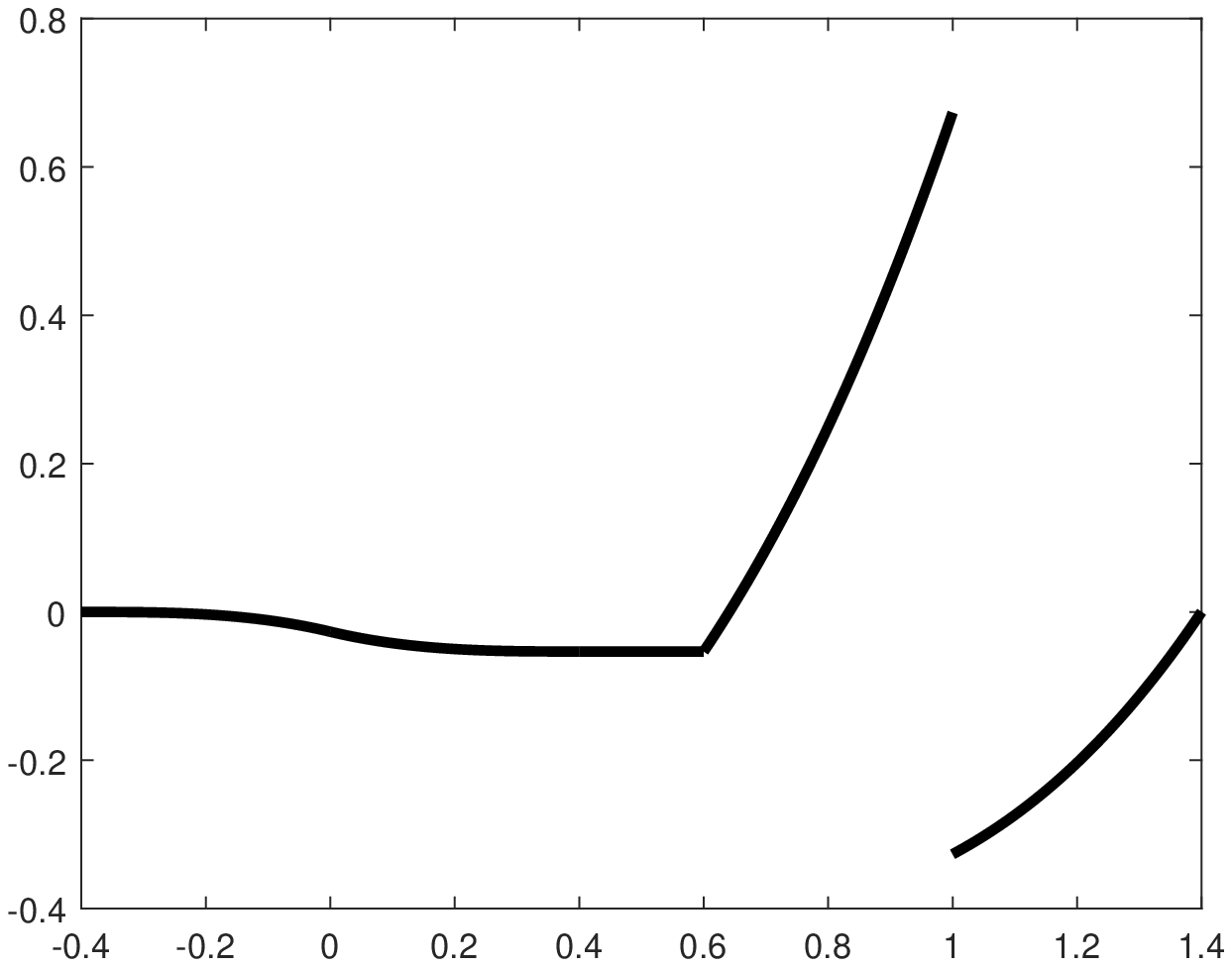}}
\subfloat[$u(x)$]{\includegraphics[width=6cm]{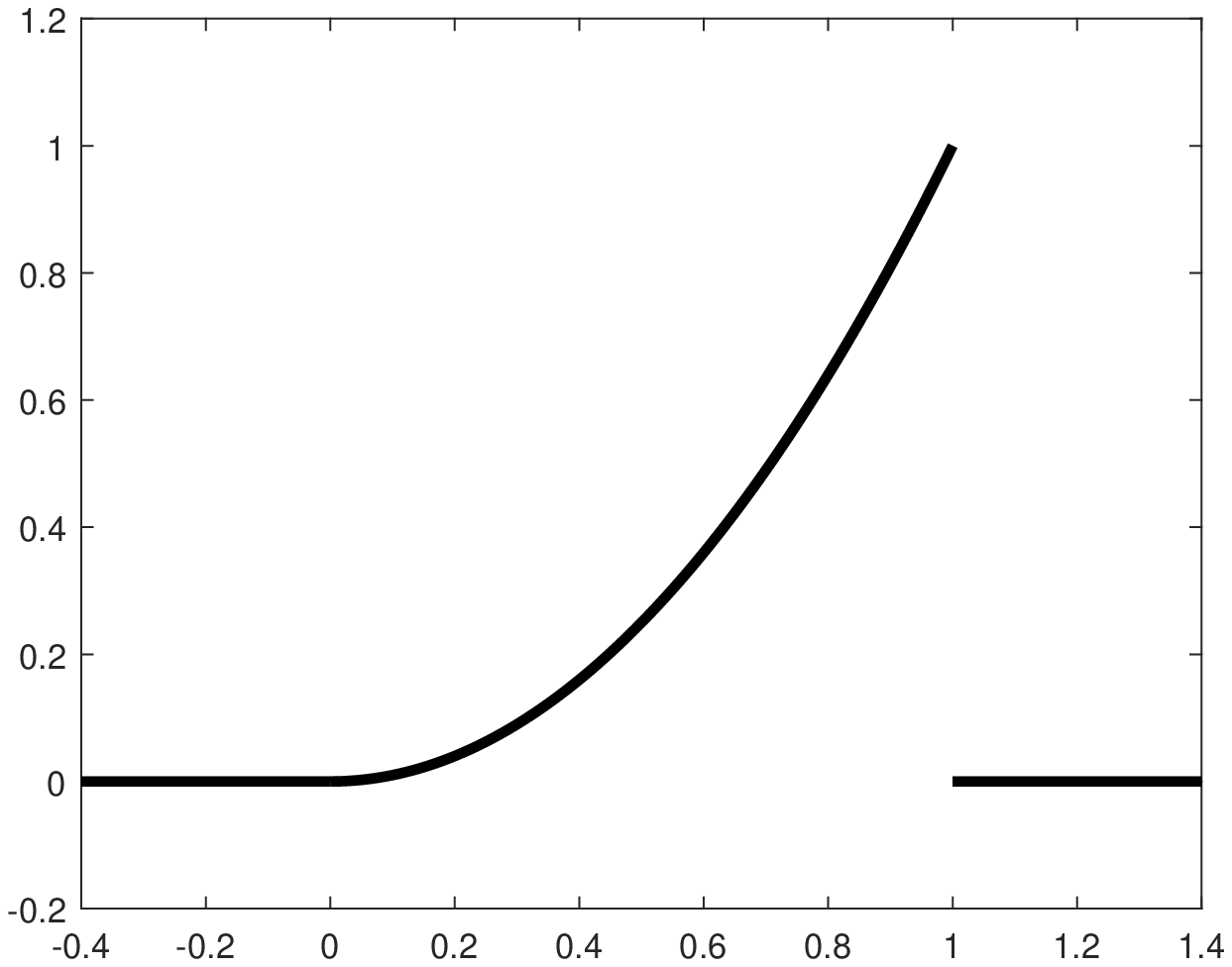}}
\caption{Example 1: The right hand side function and the exact solution.}\label{fig:1}
\end{figure}

We first use the proposed DG method on uniform meshes and conclude from Table \ref{table:ex1:uni} that convergence rates for errors in $L^2$ and $H^1$ norms are all optimal. The spectral condition number of the stiffness matrix is almost constant when the mesh size decreases, indicating the insensitivity of the discrete solution regardless how small $h$ is.
\begin{table}
\caption{Example 1: Errors in $L^2$ and $H^1$ norms, corresponding convergence rates and spectral condition numbers on uniform meshes, \emph{Cond} is abbreviation for the spectral condition number}\label{table:ex1:uni}
\begin{tabular}{lllllll}
\hline\noalign{\smallskip}
 $\delta/h$     &  4   & 8 & 16 & 32 & 64 & 128\\
 \noalign{\smallskip}\hline\noalign{\smallskip}
 $\|u-u_{h}\|$  &  7.45e-4 & 1.86e-4 & 4.66e-5 & 1.16e-5 & 2.91e-6 & 7.28e-7\\
 Rate    &  -- & 2.0000 &  2.0000 & 2.0000 & 2.0000 & 2.0000\\
 $\|u'-u'_{h}\|$ & 5.77e-2 & 2.89e-2 & 1.44e-2 & 7.22e-3 & 3.61e-3 & 1.80e-3\\
 Rate    &  -- & 1.0000 &  1.0000 & 1.0000 & 1.0000 & 1.0000\\
 Cond    &  5.8875 & 6.7743 &  6.9926 & 7.0294 & 7.0282 & 7.0225 \\
\noalign{\smallskip}\hline
\end{tabular}
\end{table}
We then use a kind of non-uniform meshes obtained by random disturbance to uniform meshes. To be specific, for a fixed $m$, let $h=\delta/m$, the non-uniform mesh is obtained by adding a random vector ${\bf \varepsilon}\in\mathbb{R}^{m-1}$ (which obeys the uniform distribution on $[-0.1h,0.1h]$) to $x_i$ to reach $x_i+\varepsilon_i,\:i=1,2,\cdots,m-1$. Together with $x_0$ and $x_m$ we get the new mesh grids
\begin{equation}\label{nonuni_grid}
x_i^n=x_i+\varepsilon_i,\:i=1,2,\cdots,m-1,\: x_0^n=x_0,\: x_m^n=x_m.
\end{equation}
We have done over twenty tests, and the convergence rates and the spectral condition numbers are all similar. Thus, instead of listing all of them, we just select one test to verify our theoretical analysis. Similar actions and presentations are made also in later examples. It is seen from Table \ref{table:ex1:non-uni} that the errors in $L^2$ and $H^1$ norms and convergence rates are comparable with that in uniform meshes case. This is consistent with the theoretical result in Theorem \ref{Thm:Convergence_rate} since $u\in C^{\infty}(\Omega)$, thus $s=2$. The spectral condition numbers of the stiffness matrices behave similar as in the uniform meshes case, too.
\begin{table}
\caption{Example 1: Errors in $L^2$ and $H^1$ norms, corresponding convergence rates and spectral condition numbers on non-uniform meshes (\ref{nonuni_grid})}\label{table:ex1:non-uni}
\begin{tabular}{llllll}
\hline\noalign{\smallskip}
 $\delta/h$     &  4   & 8 & 16 & 32 & 64\\
 \noalign{\smallskip}\hline\noalign{\smallskip}
$\|u-u_{h}\|$  &  8.04e-4 & 2.50e-4 & 6.00e-5 & 1.24e-5 & 3.14e-6  \\
Rate    &  --   & 1.6827 & 2.0629 & 2.2710 & 1.9846 \\
$\|u'-u'_{h}\|$ & 5.84e-2 & 2.95e-2 & 1.48e-2 & 7.30e-3 & 3.65e-3 \\
Rate    &  -- & 0.9848 & 0.9965 & 1.0179 & 0.9983 \\
Cond    &  5.9830 & 6.7507 & 7.0829 & 7.0786 & 7.0646\\
\noalign{\smallskip}\hline
\end{tabular}
\end{table}

\noindent{\bf Example 2.} We consider (\ref{1D_prob}) with kernel function (\ref{Con_ker}) and $b(x)=e^{x}$. The DG method we proposed in Section \ref{section:DG_method} is used to discretize it with $\delta=0.4$.
\begin{figure}
\subfloat[$b(x)$]{\includegraphics[width=6cm]{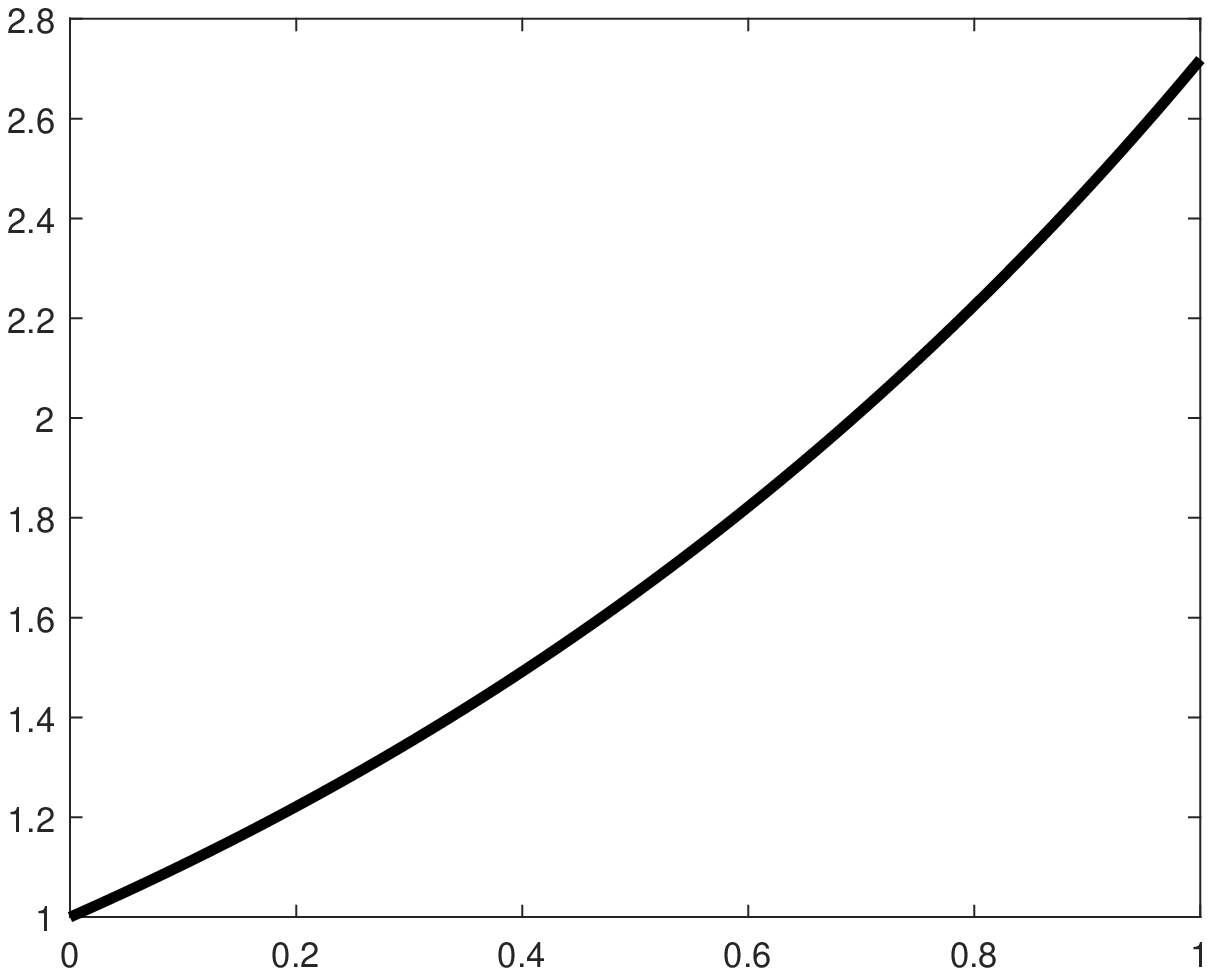}}
\subfloat[$u_{h}(x)$]{\includegraphics[width=6cm]{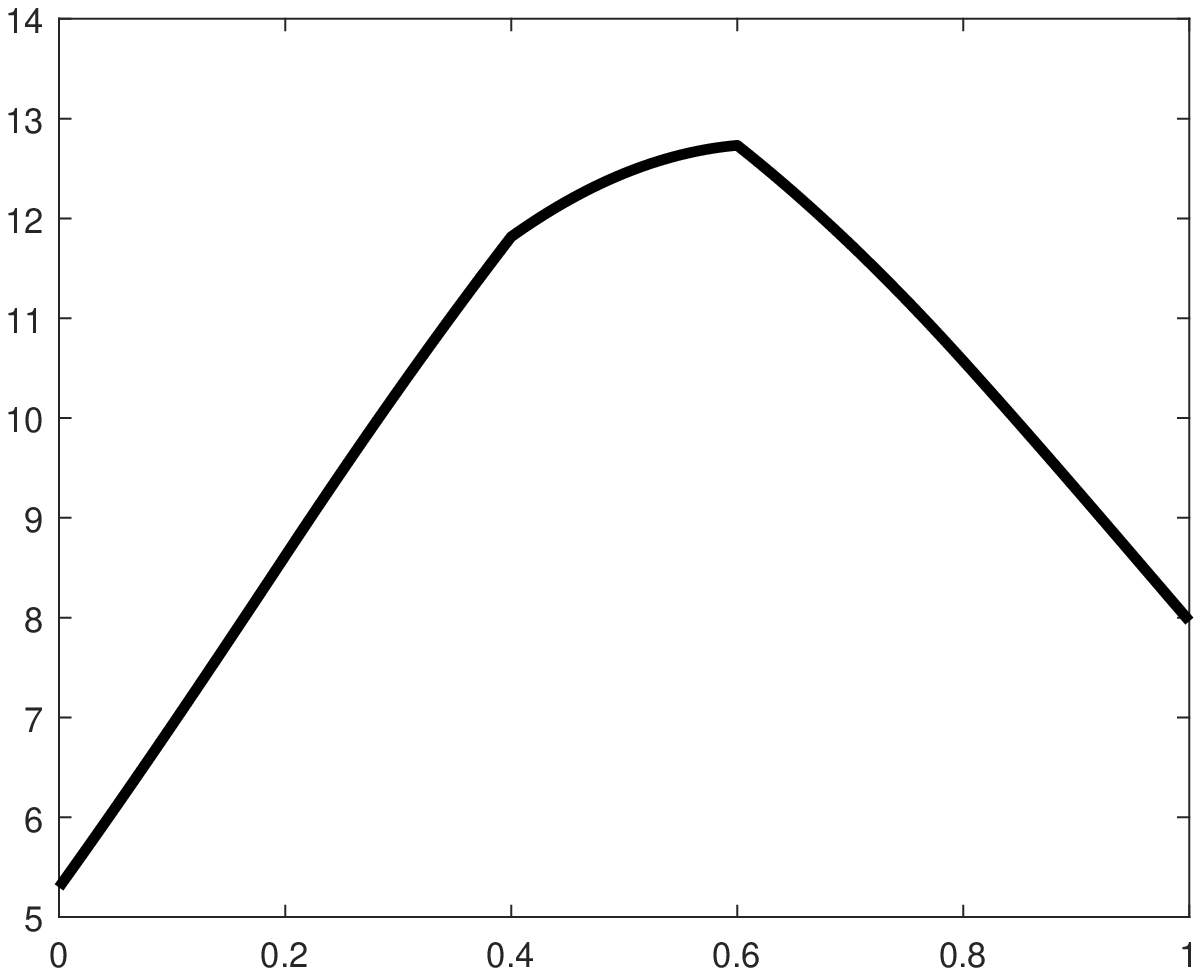}}
\caption{Example 2: The right hand side function and the approximation solution.}\label{fig:2}
\end{figure}
Since the exact solution $u(x)$ is not known for this problem we compute errors using the solution on finer meshes as approximation of the true solution. We first use the proposed DG method on uniform meshes. The right hand side function $b(x)$ and the approximation $u_{h}$ with $h=0.00625$ are plotted in Figure \ref{fig:2}. Although $b(x)$ is in $C^{\infty}(0,1)$, it has two discontinuous points $x=0$ and $x=1$, which causes the discontinuity in the first derivative of $u$ at $x=\delta=0.4$ and $x=1-\delta=0.6$. From Table \ref{table:ex2:uni} it is seen that the method has optimal convergence rates for errors in $L^2$ and $H^1$ norms. The spectral condition numbers of the stiffness matrices for the method behave similarly as Example 1. Since the results for the spectral condition numbers of the stiffness matrices are all similar in the rest of the numerical tests, we no longer list them to avoid repetitions.
\begin{table}
\caption{Example 2: Errors in $L^2$ and $H^1$ norms, corresponding convergence rates and spectral condition numbers on uniform meshes}\label{table:ex2:uni}
\begin{tabular}{llllll}
\hline\noalign{\smallskip}
 $\delta/h$     &  4   & 8 & 16 & 32 & 64\\
 \noalign{\smallskip}\hline\noalign{\smallskip}
$\|u-u_{h}\|$  & 7.54e-3  & 1.79e-3 & 4.38e-4 & 1.08e-4 & 2.69e-5 \\
Rate    &  -- & 2.0717 & 2.0349  & 2.0170 & 2.0084 \\
$\|u'-u'_{h}\|$   & 4.90e-1  & 2.41e-1 & 1.19e-1 & 5.94e-2 & 2.97e-2 \\
Rate    &  -- & 1.0259 & 1.0114  & 1.0053 & 1.0026 \\
Cond & 5.8875 & 6.7743 & 6.9926  & 7.0294 & 7.0282 \\
\noalign{\smallskip}\hline
\end{tabular}
\end{table}
Next, we use the non-uniform meshes (\ref{nonuni_grid}) to solve the problem. In this example the true solution $u(x)\in H^{1.5-\epsilon}(\Omega)$ for arbitrary small positive $\epsilon$. So for general non-uniform meshes, we expect the convergence rates for errors in $L^2$ ($H^1$) norm to be at most $1.5$ ($0.5$). This fact is verified in Table \ref{table:ex2:non-uni}. However, since we know the two discontinuous points of $u'(x)$, if they are selected as mesh grids, optimal convergence rates could be recovered. To be specific, random disturbances are added to the original mesh grids except $\delta$ and $1-\delta$. The corresponding results are shown in Table \ref{table:ex2:non-uni-grid}. That is the convergence rates for errors in $L^2$ and $H^1$ norms are $2$ and $1$, respectively, which is consistent with the theoretical result (\ref{Error_1D}). We then re-examine the computation on the meshes used in Table \ref{table:ex2:non-uni}. The convergence rates, however, are not necessarily similar, which is different with Example 1. To be specific, if the perturbation for the points $\delta$ and $1-\delta$ is large in absolute value, such as $0.1h$, the convergence rates remain similar to Table \ref{table:ex2:non-uni}. On the other hand, if the perturbation for the two points is small or close to zero, the convergence rates of the test are similar to Table \ref{table:ex2:non-uni-grid}.

\begin{table}
\caption{Example 2: Errors in $L^2$ and $H^1$ norms, and convergence rates on non-uniform meshes (\ref{nonuni_grid})}\label{table:ex2:non-uni}
\begin{tabular}{llllll}
\hline\noalign{\smallskip}
 $\delta/h$     &  4   & 8 & 16 & 32 & 64\\
 \noalign{\smallskip}\hline\noalign{\smallskip}
 $\|u-u_{h}\|$  & 9.34e-3  & 3.41e-3 & 1.24e-3 & 4.26e-4 & 1.59e-4 \\
Rate    &  -- & 1.4519 & 1.4611  & 1.5411 & 1.4222 \\
$\|u'-u'_{h}\|$   & 3.95e-1  & 2.66e-1 & 2.13e-1 & 1.42e-1 & 1.21e-1 \\
Rate    &  -- & 0.5712 & 0.3204  & 0.5815 & 0.3432 \\
\noalign{\smallskip}\hline
\end{tabular}
\end{table}

\begin{table}
\caption{Example 2: Errors in $L^2$ and $H^1$ norms, and convergence rates on non-uniform meshes (\ref{nonuni_grid}) with $\delta$ and $1-\delta$ as grids}\label{table:ex2:non-uni-grid}
\begin{tabular}{llllll}
\hline\noalign{\smallskip}
 $\delta/h$     &  4   & 8 & 16 & 32 & 64\\
 \noalign{\smallskip}\hline\noalign{\smallskip}
$\|u-u_{h}\|$  & 6.98e-3  & 1.72e-3 & 4.10e-4 & 1.01e-4 & 2.49e-5 \\
Rate    &  -- & 2.0229 & 2.0661  & 2.0243 & 2.0180 \\
$\|u'-u'_{h}\|$   & 4.17e-1  & 2.14e-1 & 1.01e-1 & 4.74e-2 & 2.42e-2 \\
Rate    &  -- & 0.9646 & 1.0781  & 1.0941 & 0.9671 \\
\noalign{\smallskip}\hline
\end{tabular}
\end{table}

\subsection{Non-constant kernel function}
We then consider the kernel function
\begin{equation}\label{Lin_ker}
     \gamma(s)=\left \{
     \begin{array}{ll}
     (1-|s|/\delta)/\delta, & |s|\leq \delta,\\
     0, & |s|>\delta.
     \end{array}
     \right .
\end{equation}
Obviously the first derivative of $\gamma$ is discontinuous at points $\pm\delta$, if $b$ is discontinuous at $x=0$ or $x=1$, thus $u$ will likely be discontinuous in its second derivative at $x=\delta$ or $x=1-\delta$. This is the case in Example 3 where the regularity pick-up is only second order and could not be further improved. That is, since $b$ is discontinuous at $x=0$ and $x=1$, the second derivative of $u$ is
is discontinuous at $x=\delta$ and $x=1-\delta$.

\noindent{\bf Example 3.} We consider (\ref{1D_prob}) with kernel function (\ref{Lin_ker}) and $b(x)=0.01e^{6x}$. The proposed DG method in Section \ref{section:DG_method} is used to discretize it with $\delta=0.4$.
\begin{figure}
\subfloat[$u_{h}(x)$]{\includegraphics[width=6cm]{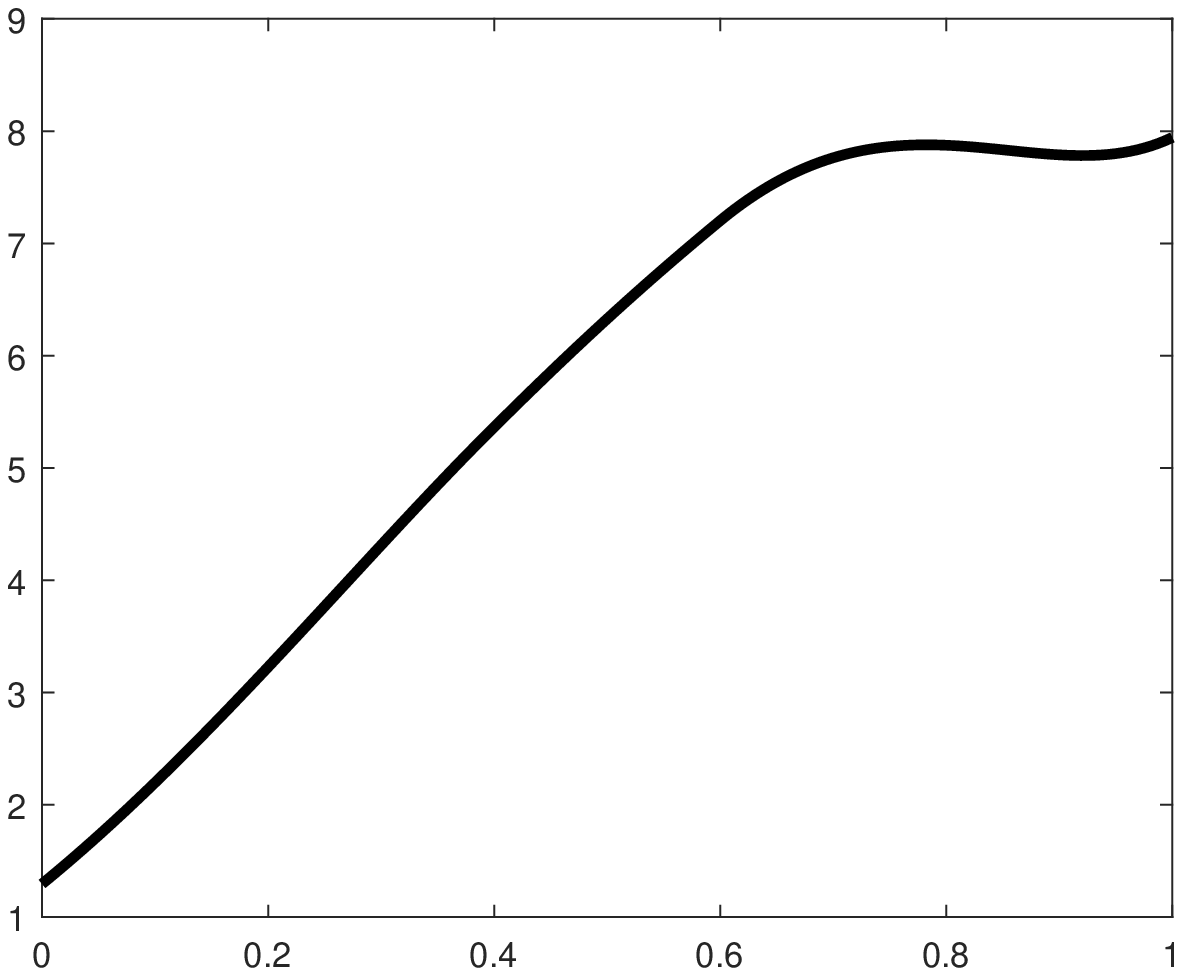}}
\subfloat[$u_{h}'(x)$]{\includegraphics[width=6cm]{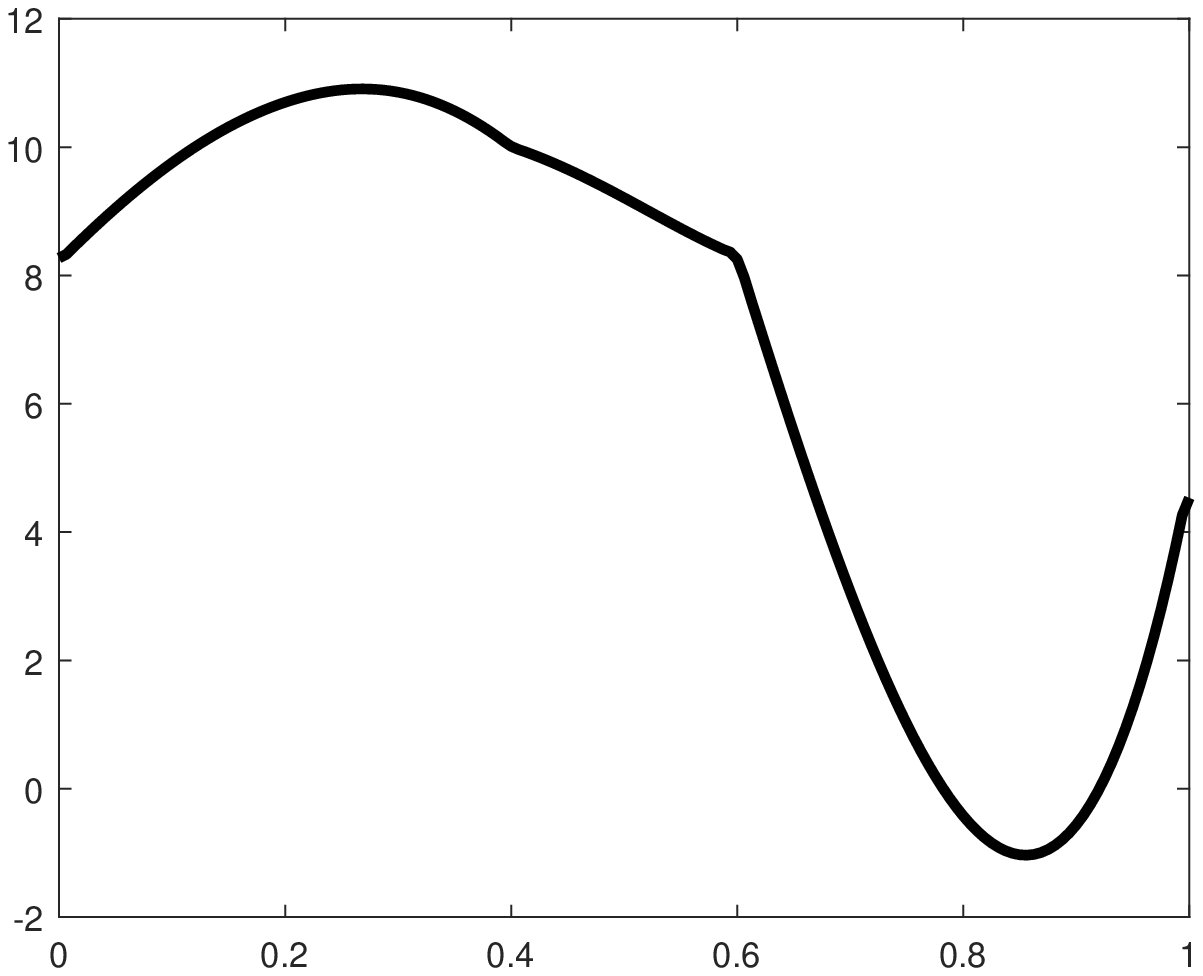}}
\caption{Example 3: The approximation solution and its derivative.}\label{fig:3}
\end{figure}
As in Example 2, since we do not know the exact solution $u(x)$, errors are computed using the solution on finer meshes as approximation of the true solution. We first implement the proposed DG method on uniform meshes. The approximation $u_{h}(x)$ with $h=0.00625$ are plotted in Figure \ref{fig:3}. To see the discontinuity of the second derivative, the first derivative of $u_{h}(x)$ is also plotted in Figure \ref{fig:3}. Although $b(x)\in C^{\infty}(0,1)$, however, $b(x)$ has two discontinuous points $x=0$ and $x=1$, which causes the discontinuity for the second derivative of $u$ at $x=\delta=0.4$ and $x=1-\delta=0.6$. From Table \ref{table:ex3:uni} it is seen that optimal convergence rates for errors in $L^2$ and $H^1$ norms are achieved.

\begin{table}
\caption{Example 3: Errors in $L^2$ and $H^1$ norms, and convergence rates on uniform meshes}\label{table:ex3:uni}
\begin{tabular}{llllll}
\hline\noalign{\smallskip}
 $\delta/h$     &  4   & 8 & 16 & 32 & 64\\
 \noalign{\smallskip}\hline\noalign{\smallskip}
$\|u-u_{h}\|$  &  1.18e-2 & 2.74e-3  & 6.61e-4 & 1.62e-4 & 4.03e-5 \\
Rate           &  -- & 2.1019 & 2.0525  & 2.0246 & 2.0116 \\
$\|u'-u'_{h}\|$ & 7.20e-1  & 3.60e-1 & 1.79e-1 & 8.90e-2 & 4.44e-2 \\
Rate    &  -- & 1.0018 & 1.0089  & 1.0056 & 1.0030 \\
\noalign{\smallskip}\hline
\end{tabular}
\end{table}
Next, we use non-uniform meshes (\ref{nonuni_grid}) to solve the problem. In this example the true solution satisfies $u(x)\in H^{2.5-\epsilon}(\Omega)$ for arbitrary small positive $\epsilon$. So for general non-uniform meshes the theoretical convergence rates for errors in $L^2$ and $H^1$ norms are all optimal. This is indeed verified in Table \ref{table:ex3:non-uni}.
\begin{table}
\caption{Example 3: Errors in $L^2$ and $H^1$ norms, and convergence rates on nonuniform meshes (\ref{nonuni_grid})}\label{table:ex3:non-uni}
\begin{tabular}{llllll}
\hline\noalign{\smallskip}
 $\delta/h$     &  4   & 8 & 16 & 32 & 64\\
 \noalign{\smallskip}\hline\noalign{\smallskip}
$\|u-u_{h}\|$  &  1.11e-2 & 2.61e-3  & 6.08e-4 & 1.51e-4 & 3.76e-5 \\
Rate           &  -- & 2.0891 & 2.0993  & 2.0143 & 1.9985 \\
$\|u'-u'_{h}\|$   & 6.70e-1  & 3.32e-1 & 1.51e-1 & 7.47e-2 & 3.67e-2 \\
Rate    &  -- & 1.0128 & 1.1331  & 1.0189 & 1.0261 \\
\noalign{\smallskip}\hline
\end{tabular}
\end{table}

\begin{remark}
We have considered the homogeneous problem (\ref{nonlocal_diff_homo}) with a right hand side function $b({\bf x})\in C(\Omega)$. The inhomogeneous problem (\ref{nonlocal_diffusion}) could be converted to a homogeneous one like (\ref{nonlocal_diff_homo}) via the transforms (\ref{btrans}) and (\ref{utrans}).
\end{remark}

\section{Conclusion}
We propose a new kind of DG method in this paper to numerically solve the nonlocal models with integrable kernels. The existed references tell us that if the right hand side function and the volume constraint function are in $L^{2}(\Omega)$ and $L^{2}(\Omega_{I})$, respectively, then the true solution of that nonlocal model also belongs to $L^{2}(\Omega\cup\Omega_{I})$. Such a general result makes the numerical approximation difficult to operate, or easy to operate but not so efficiently. To make the approximation easier and more efficient simultaneously, we first convert the original nonhomogeneous problem with right hand side function in $L^{2}(\Omega)$ to a homogeneous problem with right hand side function continuous in $\Omega$. Then we analyze the structure of true solution of the homogeneous problem, especially for higher dimensional cases. The main result is, this kind of problem often encounters the discontinuity across the boundary $\partial\Omega$, thus possibly causes the discontinuity of first or second derivative (perhaps higher order derivatives, depending on the smoothness of the kernels) across $\partial\Omega_{1}$. Based on this observation, an appropriate DG method is proposed which has some good properties, such as, the matrix of the algebraic system is symmetrical positive definite and has almost constant spectral condition number independent of the mesh size, the method is asymptotically compatible and uses less degrees of freedom compared with direct use of DG method. Moreover, it has optimal convergence rate for 1D case under very weak assumption, and the almost optimal convergence rate for 2D case under mild assumption. This is the essential improvement over the existed theory for standard approximation like continuous FEM.

The error and condition number estimates for the method are proven in Section \ref{section:Theoretical considerations} for any dimensional case. However, the numerical experiments are implemented just for 1D case in Section \ref{section:numerical_experiment}. This is because the implementation of a FEM for a nonlocal problem involves calculation of double integral which is rather complicated for higher dimensional cases. To be more specific, since the kernels we discussed are supported on a ball of radius $\delta$, when the inner integral of the double integral is written by a sum of some components (also integrals) over the elements which overlap with the support of the kernel, the integral region for some components is strictly contained within the support of the kernel. For such special integrals, some specifically designed quadrature rule should be used to obtain good accuracy. Authors in \cite{xu2015multiscale,xu2016multiscale} discuss this issue in details and use the quadrature rule for multiscale implementation for PD models. Interested readers are referred to them and references cited therein.

\bibliographystyle{plain}
\bibliography{references}

\end{document}